\documentclass[11pt,a4paper]{article}
\usepackage[top=2.5cm,bottom=2.5cm,left=2.2cm,right=2.2cm]{geometry}
\usepackage[T1]{fontenc}
\usepackage[utf8]{inputenc}
\usepackage{color}
\usepackage{graphicx}
\usepackage{amsfonts}
\usepackage{extarrows}
\usepackage{amsmath,amsthm,amssymb,color}
\usepackage{hyperref}
\usepackage{eepic}
\usepackage{lineno}
\usepackage{enumerate}	
\usepackage{paralist}
\usepackage{cite}
\usepackage{algorithm}
\usepackage{authblk}
\usepackage{mathtools}
\usepackage{pifont}
\usepackage{tikz}
\usepackage{algorithmic}
\usepackage{subcaption}

\usetikzlibrary{decorations.pathreplacing}

\hypersetup
{
	colorlinks=true,
	linkcolor=blue,
	filecolor=blue,
	urlcolor=blue,
	citecolor=cyan,
}

\newtheorem{theorem}{Theorem}[section]
\newtheorem{proposition}[theorem]{Proposition}

\newtheorem{corollary}[theorem]{Corollary}
\newtheorem{conjecture}[theorem]{Conjecture}

\usepackage{etoolbox}
\theoremstyle{definition}

\newtheorem{claim}{\indent Claim}

\newtheorem{question}{Question}

\AtBeginEnvironment{proof}{\setcounter{case}{0}}

\title{Three-coloring triangle-free graphs without long forbidden paths}

\author[a]{Yidong Zhou}
\author[b]{Jorik Jooken}
\author[a]{Baoyuan Shan}
\author[b,c]{Jan Goedgebeur}
\author[d]{Shenwei Huang\thanks{Corresponding author: shenweihuang@nankai.edu.cn.}}

\affil[a]{College of Computer Science, Nankai University, Tianjin 300350, China}
\affil[b]{Department of Computer Science, KU Leuven Campus Kulak-Kortrijk, 8500 Kortrijk, Belgium}
\affil[c]{Department of Mathematics, Computer Science and Statistics, Ghent University, 9000 Ghent, Belgium}
\affil[d]{School of Mathematical Sciences and LPMC, Nankai University, Tianjin 300071, China}

\begin{document}

\maketitle

\begin{abstract}
A graph $G$ is $k$-vertex-critical if $\chi(G)=k$, but $\chi(G')<k$ for every proper induced subgraph $G'$ of $G$. For a family of graphs $\mathcal{F}$, $G$ is $\mathcal{F}$-free if no graph $F \in \mathcal{F}$ is an induced subgraph of $G$. We show that there are exactly three 4-vertex-critical $\{P_7,C_3\}$-free graphs containing an induced $C_7$, thereby settling the first of the two cases of a conjecture by Goedgebeur and Schaudt [J.~Graph Theory, 87:188--207, 2018]. Moreover, we show that all $\{P_5+P_1,C_3\}$-free graphs are $3$-colorable and by combining our result with known results from the literature, we completely characterize the maximum chromatic number of $\{F,C_3\}$-free graphs if $F$ is a six-vertex induced subgraph of $P_7$. Finally, we construct an infinite family of $4$-vertex-critical $\{4K_2,C_3\}$-free graphs. These graphs are also $\{P_{11},C_3\}$-free and this is the first value of $t$ for which an infinite family of $4$-vertex-critical $\{P_{t},C_3\}$-free graphs is known.
\end{abstract}


\section{Introduction}
Given a graph $G$ and an integer $k$, we say that $G$ admits a \textit{proper $k$-coloring} (or alternatively, $G$ is \textit{$k$-colorable}) if it is possible to assign a label from the set $\{1,2,\ldots,k\}$ to each vertex in $G$ such that neighboring vertices receive a different label. The \textit{chromatic number $\chi(G)$} is the smallest integer $k$ such that $G$ admits a proper $k$-coloring. For a fixed integer $k$, the $k$-coloring problem receives as input a graph $G$ and asks whether $G$ admits a proper $k$-coloring. It is well known that for any $k \geq 3$ the $k$-coloring problem is NP-complete~\cite{K72}. This motivates the study of this problem under additional restrictions on the input graphs. In this setting, one is interested in investigating which restrictions make the problem tractable. For graphs $G$ and $F$, we say that $F$ is an \textit{induced subgraph} of $G$ if $F$ can be obtained by removing zero or more vertices from $G$. We call $F$ a \textit{proper induced subgraph} of $G$ if $F$ is an induced subgraph of $G$ different from $G$ itself. If $F$ is not an induced subgraph of $G$, we say that $G$ is \textit{$F$-free}. Similarly, for a family of graphs $\mathcal{F}$, we say that $G$ is \textit{$\mathcal{F}$-free} if $G$ is $F$-free for each $F \in \mathcal{F}$. A popular restriction in the literature is to only consider input graphs that are $\mathcal{F}$-free for some fixed family of graphs $\mathcal{F}$, see e.g.\ the survey~\cite{GJPS17}. For an integer $t$, the graph $P_t$ denotes a path on $t$ vertices, whereas the graph $C_t$ denotes a cycle on $t$ vertices. For graphs $G$ and $H$, the graph $G+H$ denotes the disjoint union of $G$ and $H$, whereas for an integer $\ell \geq 1$ the graph $\ell G$ denotes the disjoint union of $\ell$ copies of $G$.

If $\mathcal{F}$ contains exactly one graph $F$, it is known that for each integer $k \geq 3$, the $k$-coloring problem restricted to $F$-free graphs remains NP-complete unless $F$ is a disjoint union of paths~\cite{EHK98,HI81,LD83}. The case where $F$ is connected (i.e., $F$ is a path) received the most attention. Huang \cite{HS13} proved that the 4-coloring problem for $P_7$-free graphs and the 5-coloring problem for $P_6$-free graphs are both NP-complete. On the other hand, Ho\`{a}ng, Kami\'{n}ski, Lozin, Sawada and Shu \cite{HKLSS10} proved that the $k$-coloring problem for $P_5$-free graphs can be solved in polynomial time for any fixed $k\geq 5$, whereas Chudnovsky, Spirkl and Zhong \cite{CM24,CM241} showed that the $4$-coloring problem for $P_6$-free graphs can be solved in polynomial time. By combining these results, one obtains a complete classification of the complexity of $k$-coloring $P_t$-free graphs for any fixed $k\geq 4$ and $t \geq 1$. 
For $k=3$, the strongest known result is due to Bonomo, Chudnovsky, Maceli, Schaudt, Stein and Zhong \cite{BCMSSZ18} who proved that the 3-coloring problem for $P_7$-free graphs can be solved in polynomial time.

A more fine-grained view on the cases that can be solved in polynomial time can be obtained through the lens of \textit{$k$-vertex-critical graphs}, which we will discuss next. A graph $G$ is called $k$-vertex-critical if $\chi(G)=k$, but $\chi(G')<k$ for every proper induced subgraph $G'$ of $G$. For example, the family of all $3$-vertex-critical graphs is the family consisting of all cycles of odd length. For any family of graphs $\mathcal{F}$, a folklore result is that if the family $\mathcal{S}_{\mathcal{F},k}$ consisting of all $k$-vertex-critical $\mathcal{F}$-free graphs is a finite family, then the $(k-1)$-coloring problem on $\mathcal{F}$-free graphs can be solved in polynomial time. Indeed, to determine whether an $\mathcal{F}$-free graph $G$ is $(k-1)$-colorable or not, it suffices to iterate over each $k$-vertex-critcal graph $G' \in \mathcal{S}_{\mathcal{F},k}$ and determine whether $G'$ is an induced subgraph of $G$ or not. If $\mathcal{S}_{\mathcal{F},k}$ is a finite family, the order of each graph $G' \in \mathcal{S}_{\mathcal{F},k}$ is bounded by some constant and therefore this can be done in polynomial time. If some $G' \in \mathcal{S}_{\mathcal{F},k}$ occurs as an induced subgraph of $G$, the algorithm returns that $G$ is not $(k-1)$-colorable and can also return $G'$ as a proof that certifies that $G$ is not $(k-1)$-colorable and otherwise the algorithm can safely return that $G$ must be $(k-1)$-colorable. Hence, the statement that $\mathcal{S}_{\mathcal{F},k}$ is a finite family is at least as strong as the statement that the $(k-1)$-coloring problem on $\mathcal{F}$-free graphs can be solved in polynomial time, since the latter statement does not guarantee that the algorithm can also provide a certificate explaining why the input graph is not $(k-1)$-colorable.

Chudnovsky, Goedgebeur, Schaudt and Zhong~\cite{CGSZ20b} showed that there are only finitely many $4$-vertex-critical $P_6$-free graphs (and give the complete list of $80$ such graphs), but infinitely many $4$-vertex-critical $P_7$-free graphs. The same authors later complemented this result by presenting a complete characterisation of when there are only finitely many 4-vertex-critical $\mathcal{F}$-free graphs if $\mathcal{F}$ contains exactly one graph $F$: there are only finitely many 4-vertex-critical $F$-free graphs if and only if $F$ is an induced subgraph of $P_6$, $2P_3$ or $P_4+\ell P_1$ for some integer $\ell \geq 1$~\cite{CGSZ20a}. This result led several authors to investigate the case when $\mathcal{F}$ contains two graphs, namely a path and another graph. Bonomo-Braberman et al.~\cite{BEtAl21} showed that there are only finitely many 4-vertex-critical $\{P_7,C_3\}$-free graphs (more precisely, the order of such a graph is smaller than 201\,326\,592), while Goedgebeur and Schaudt~\cite{GS18} made the following more precise conjecture.
\begin{conjecture}[Goedgebeur and Schaudt~\cite{GS18}]
\label{conj:GS}
    There are exactly seven 4-vertex-critical $\{P_7,C_3\}$-free graphs.
\end{conjecture}
The seven graphs from Conjecture~\ref{conj:GS} are shown in Fig.~\ref{fig:7Graphs}.

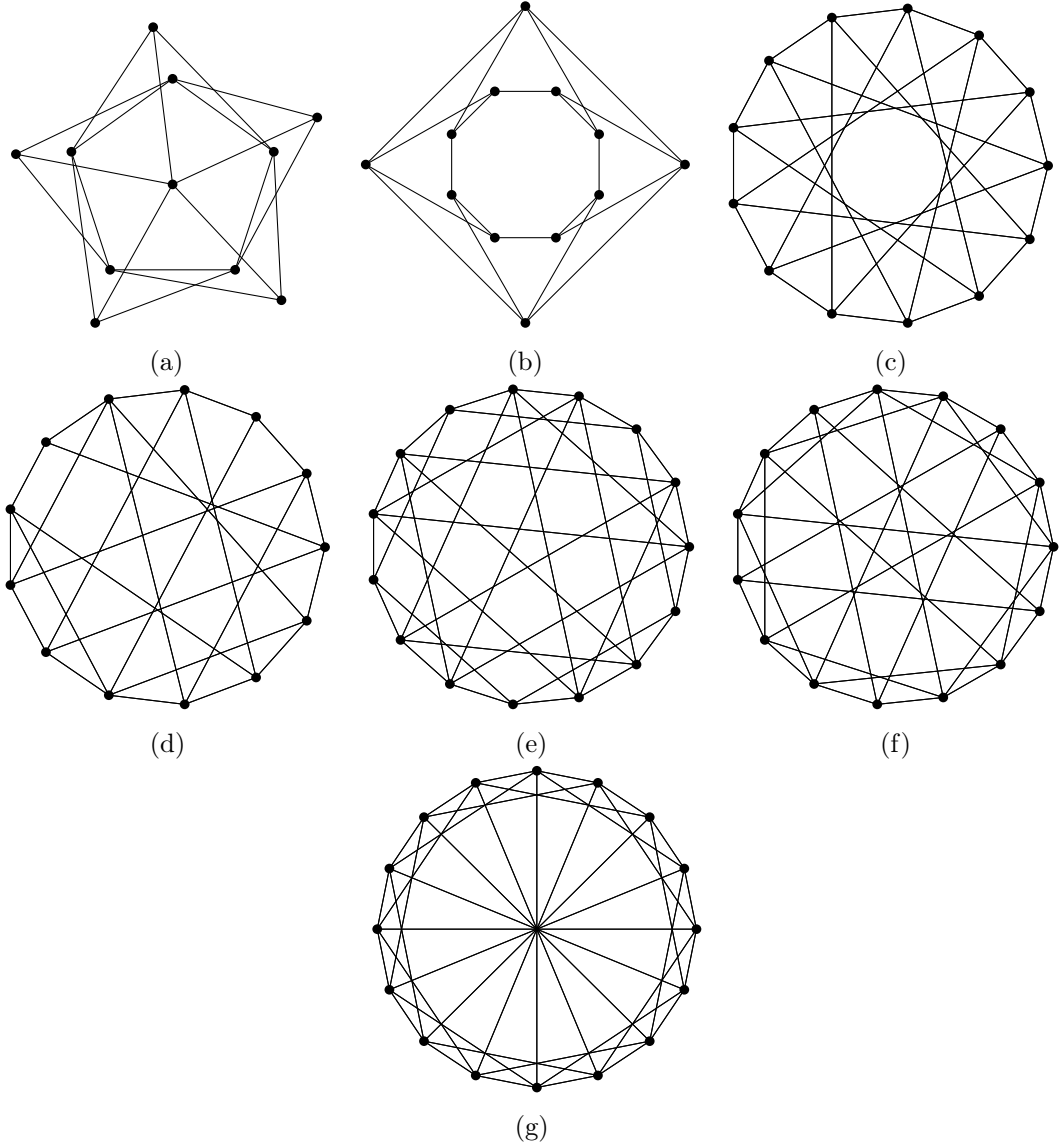
\begin{figure}[h!]
\begin{center}
\subcaptionbox{}{\begin{tikzpicture}[scale=0.7]
  \def\sides{5}
  \def\radiusA{2}
  \def\radiusB{3}

  \fill (0,0) circle (2.61pt);  
  \foreach \i in {1,...,\sides} {
    \fill ({360/\sides * \i+18}:\radiusA) circle (2.61pt);
    \fill ({360/\sides * \i+25}:\radiusB) circle (2.61pt);

    \draw ({360/\sides * (\i + 1)+18}:\radiusA) -- ({360/\sides * \i+18}:\radiusA);

    \draw ({360/\sides * (\i + 1)+25}:\radiusB) -- ({360/\sides * \i+18}:\radiusA);
    \draw ({360/\sides * (\i - 1)+25}:\radiusB) -- ({360/\sides * \i+18}:\radiusA);

    \draw ({360/\sides * (\i)+25}:\radiusB) -- (0,0);
  }
  
\end{tikzpicture}} \quad
\subcaptionbox{}{\begin{tikzpicture}[scale=0.7]
  \def\sidesA{8}
  \def\sidesB{4}
  \def\radiusA{1.5}
  \def\radiusB{3}

  \foreach \i in {1,...,\sidesA} {
    \fill ({360/\sidesA * \i+22.5}:\radiusA) circle (2.61pt);

    \draw ({360/\sidesA * (\i + 1)+22.5}:\radiusA) -- ({360/\sidesA * \i+22.5}:\radiusA);

  }

  \foreach \i in {1,...,\sidesB} {
    \fill ({360/\sidesB * \i}:\radiusB) circle (2.61pt);

    \draw ({360/\sidesB * (\i + 1)}:\radiusB) -- ({360/\sidesB * \i}:\radiusB);

    \draw ({360/\sidesA * (2*\i + 6)+22.5}:\radiusA) -- ({360/\sidesB * \i}:\radiusB);
    \draw ({360/\sidesA * (2*\i + 1)+22.5}:\radiusA) -- ({360/\sidesB * \i}:\radiusB);

  }
\end{tikzpicture}} \quad
\subcaptionbox{}{\begin{tikzpicture}[scale=0.7]
  \def\sides{13}
  \def\radius{3}

  \foreach \i in {1,...,\sides} {
    \fill ({360/\sides * \i}:\radius) circle (2.61pt);

    \draw ({360/\sides * (\i + 1)}:\radius) -- ({360/\sides * \i}:\radius);

    \draw ({360/\sides * (\i + 5)}:\radius) -- ({360/\sides * \i}:\radius);

    \draw ({360/\sides * (\i + 8)}:\radius) -- ({360/\sides * \i}:\radius);

    \draw ({360/\sides * (\i + 12)}:\radius) -- ({360/\sides * \i}:\radius);
  }
\end{tikzpicture}} \quad
\subcaptionbox{}{\begin{tikzpicture}[scale=0.7]
  \def\sides{13}
  \def\radius{3}

  \foreach \i in {1,...,\sides} {
    \fill ({360/\sides * \i}:\radius) circle (2.61pt);

    \draw ({360/\sides * (\i + 1)}:\radius) -- ({360/\sides * \i}:\radius);

    \draw ({360/\sides * (\i + 12)}:\radius) -- ({360/\sides * \i}:\radius);
  }

  \foreach \i in {1} {
    \draw ({360/\sides * (\i + 6)}:\radius) -- ({360/\sides * \i}:\radius);

    \draw ({360/\sides * (\i + 9)}:\radius) -- ({360/\sides * \i}:\radius);
  }
  \foreach \i in {2} {
    \draw ({360/\sides * (\i + 7)}:\radius) -- ({360/\sides * \i}:\radius);
  }
  \foreach \i in {3,8,11,13} {
    \draw ({360/\sides * (\i + 5)}:\radius) -- ({360/\sides * \i}:\radius);

    \draw ({360/\sides * (\i + 8)}:\radius) -- ({360/\sides * \i}:\radius);
  }
  \foreach \i in {4} {
    \draw ({360/\sides * (\i + 3)}:\radius) -- ({360/\sides * \i}:\radius);

    \draw ({360/\sides * (\i + 6)}:\radius) -- ({360/\sides * \i}:\radius);
    
    \draw ({360/\sides * (\i + 8)}:\radius) -- ({360/\sides * \i}:\radius);
  }
  \foreach \i in {5} {
    \draw ({360/\sides * (\i + 8)}:\radius) -- ({360/\sides * \i}:\radius);
  }
  \foreach \i in {6} {
    \draw ({360/\sides * (\i + 3)}:\radius) -- ({360/\sides * \i}:\radius);

    \draw ({360/\sides * (\i + 5)}:\radius) -- ({360/\sides * \i}:\radius);
  }
  \foreach \i in {7} {
    \draw ({360/\sides * (\i + 7)}:\radius) -- ({360/\sides * \i}:\radius);

    \draw ({360/\sides * (\i + 10)}:\radius) -- ({360/\sides * \i}:\radius);
  }
  \foreach \i in {9} {
    \draw ({360/\sides * (\i + 3)}:\radius) -- ({360/\sides * \i}:\radius);

    \draw ({360/\sides * (\i + 6)}:\radius) -- ({360/\sides * \i}:\radius);
    
    \draw ({360/\sides * (\i + 10)}:\radius) -- ({360/\sides * \i}:\radius);
  }
  \foreach \i in {10} {
    \draw ({360/\sides * (\i + 4)}:\radius) -- ({360/\sides * \i}:\radius);

    \draw ({360/\sides * (\i + 7)}:\radius) -- ({360/\sides * \i}:\radius);
  }
  \foreach \i in {12} {
    \draw ({360/\sides * (\i + 5)}:\radius) -- ({360/\sides * \i}:\radius);

    \draw ({360/\sides * (\i + 10)}:\radius) -- ({360/\sides * \i}:\radius);
  }
\end{tikzpicture}} \quad
\subcaptionbox{}{\begin{tikzpicture}[scale=0.7]
  \def\sides{15}
  \def\radius{3}

  \foreach \i in {1,...,\sides} {
    \fill ({360/\sides * \i}:\radius) circle (2.61pt);

    \draw ({360/\sides * (\i + 1)}:\radius) -- ({360/\sides * \i}:\radius);

    \draw ({360/\sides * (\i + 14)}:\radius) -- ({360/\sides * \i}:\radius);
  }

  \foreach \i in {1,4,...,\sides} {
    \draw ({360/\sides * (\i + 5)}:\radius) -- ({360/\sides * \i}:\radius);

    \draw ({360/\sides * (\i + 8)}:\radius) -- ({360/\sides * \i}:\radius);

    \draw ({360/\sides * (\i + 11)}:\radius) -- ({360/\sides * \i}:\radius);
  }

  \foreach \i in {2,5,...,\sides} {
    \draw ({360/\sides * (\i + 3)}:\radius) -- ({360/\sides * \i}:\radius);

    \draw ({360/\sides * (\i + 12)}:\radius) -- ({360/\sides * \i}:\radius);
  }

    \foreach \i in {3,6,...,\sides} {
    \draw ({360/\sides * (\i + 4)}:\radius) -- ({360/\sides * \i}:\radius);

    \draw ({360/\sides * (\i + 7)}:\radius) -- ({360/\sides * \i}:\radius);

    \draw ({360/\sides * (\i + 10)}:\radius) -- ({360/\sides * \i}:\radius);
  }
\end{tikzpicture}} \quad
\subcaptionbox{}{\begin{tikzpicture}[scale=0.7]
  \def\sides{15}
  \def\radius{3}

  \foreach \i in {1,...,\sides} {
    \fill ({360/\sides * \i}:\radius) circle (2.61pt);

    \draw ({360/\sides * (\i + 1)}:\radius) -- ({360/\sides * \i}:\radius);

    \draw ({360/\sides * (\i + 14)}:\radius) -- ({360/\sides * \i}:\radius);
  }

  \foreach \i in {1,4,...,\sides} {
    \draw ({360/\sides * (\i + 3)}:\radius) -- ({360/\sides * \i}:\radius);

    \draw ({360/\sides * (\i + 8)}:\radius) -- ({360/\sides * \i}:\radius);

    \draw ({360/\sides * (\i + 12)}:\radius) -- ({360/\sides * \i}:\radius);
  }

  \foreach \i in {2,5,...,\sides} {
    \draw ({360/\sides * (\i + 6)}:\radius) -- ({360/\sides * \i}:\radius);

    \draw ({360/\sides * (\i + 9)}:\radius) -- ({360/\sides * \i}:\radius);
  }

    \foreach \i in {3,6,...,\sides} {
    \draw ({360/\sides * (\i + 3)}:\radius) -- ({360/\sides * \i}:\radius);

    \draw ({360/\sides * (\i + 7)}:\radius) -- ({360/\sides * \i}:\radius);

    \draw ({360/\sides * (\i + 12)}:\radius) -- ({360/\sides * \i}:\radius);
  }
\end{tikzpicture}} \quad
\subcaptionbox{}{
\begin{tikzpicture}[scale=0.7]
  \def\sides{16}
  \def\radius{3}

  \foreach \i in {1,...,\sides} {
    \fill ({360/\sides * \i}:\radius) circle (2.61pt);

    \draw ({360/\sides * (\i + 1)}:\radius) -- ({360/\sides * \i}:\radius);

    \draw ({360/\sides * (\i + 3)}:\radius) -- ({360/\sides * \i}:\radius);
    \draw ({360/\sides * (\i + 8)}:\radius) -- ({360/\sides * \i}:\radius);
    \draw ({360/\sides * (\i + 13)}:\radius) -- ({360/\sides * \i}:\radius);
    
    \draw ({360/\sides * (\i + 15)}:\radius) -- ({360/\sides * \i}:\radius);
  }

\end{tikzpicture}}
\end{center}
\caption{The seven 4-vertex-critical $\{P_7,C_3\}$-free graphs from Conjecture~\ref{conj:GS}.}\label{fig:7Graphs} 
\end{figure}

Since every 4-vertex-critical $\{P_7,C_3\}$-free graph is non-bipartite, it must contain an induced $C_5$ or an induced $C_7$. The first contribution of the current paper is that we make progress on Conjecture~\ref{conj:GS} by settling one of the two cases: we show that there are exactly three 4-vertex-critical $\{P_7,C_3\}$-free graphs containing an induced $C_7$, namely graphs (d), (e) and (f) from Fig.~\ref{fig:7Graphs}. This result is partially computer-assisted. We describe the key algorithm that was used in Section~\ref{sec:algo} and the theoretical proof in which some parts rely on this algorithm in Section~\ref{sec:inducedC7}. 

Related to this, several authors have also focused on the case where $P_7$ is replaced by a six-vertex induced subgraph of $P_7$ (i.e., $P_6$, $P_5+P_1$, $P_4+P_2$ or $2P_3$). As we mentioned before, there are only finitely many 4-vertex-critical $2P_3$-free graphs~\cite{CGSZ20a} and finitely many 4-vertex-critical $P_6$-free graphs~\cite{CGSZ20b}. Pyatkin~\cite{P13} and Randerath, Schiermeyer and Tewes~\cite{RST02} showed that $\chi(G) \leq 4$ for every $\{2P_3,C_3\}$-free and every $\{P_6,C_3\}$-free graph $G$, respectively. Recently, Chen, Wu and Zhang~\cite{CWZ25} showed that there is exactly one 4-vertex-critical $\{P_4+P_2,C_3\}$-free graph and that $\chi(G) \leq 4$ for every $\{P_4+P_2,C_3\}$-free graph $G$. In Section~\ref{sec:inducedC7}, we present the second contribution of the current paper where we solve the last remaining case (thereby obtaining a complete characterization for the $\{F,C_3\}$-free case when $F$ is a six-vertex induced subgraph of $P_7$) by showing that $\chi(G) \leq 3$ for every $\{P_5+P_1,C_3\}$-free graph $G$ (and thus, there are no 4-vertex-critical $\{P_5+P_1,C_3\}$-free graphs).

The \textit{girth} of a graph $G$ is the length of its shortest cycle (or $\infty$ if $G$ has no cycles). Hence, graphs with girth at least $g$ correspond with $\{C_3,C_4,\ldots,C_{g-1}\}$-free graphs, which generalizes the $C_3$-free setting we considered so far. For any positive integer $t$, let $g(t)$ be the minimum positive integer such that there are only finitely many 4-vertex-critical $P_t$-free graphs with girth at least $g(t)$. The results from Kami\'{n}ski and Pstrucha~\cite{KP19} imply that for all $t \geq 1$, we have $g(t) \leq 5$. There are only few values of $t$ for which the exact value of $g(t)$ is known. From the results mentioned in our previous discussion, it follows that $g(t)=3$ when $1 \leq t \leq 6$~\cite{CGSZ20b} and $g(7)=4$~\cite{CGSZ20b,BEtAl21}. In Section~\ref{sec:infFam} we present the third contribution of the current paper and construct an infinite family of 4-vertex-critical $\{4K_2, C_3\}$-free graphs (note that these graphs are also $\{P_{11}, C_3\}$-free). This is the first infinite family of 4-vertex-critical $\{P_t, C_3\}$-free graphs for any value of $t$. Hence, by combining our result with the result by Kami\'{n}ski and Pstrucha~\cite{KP19}, we show that $g(t)=5$ for all $t \geq 11$.
\subsection{Further notation}

For an integer $n \geq 1$, we write $[n]$ to denote the set $\{1,2,\ldots,n\}$. Let $G$ be a graph and $A, B \subseteq V(G)$ be two sets containing vertices of $G$. We write $N_G(A) := \{u \in V(G) \setminus A~|~ \exists v \in A\text{ with }uv\in E(G)\}$ (or $N(A)$ if $G$ is clear from the context) to denote all neighbors of vertices in $A$ that lie outside $A$. We say that a vertex $v \in V(G)$ is complete to $A$ if $v$ is adjacent to every vertex in $A$. We say that $A$ is complete to $B$ if every vertex $v \in A$ is complete to $B$. Similarly, we say that a vertex $v \in V(G)$ is anti-complete to $A$ if $v$ is not adjacent to any vertex in $A$ and $A$ is anti-complete to $B$ if every vertex $v \in A$ is anti-complete to $B$. The notation $G[A]$ denotes the graph induced by the vertices in $A$. In several places in the paper, we consider labeled objects (such as vertices) $v_1, v_2, \ldots, v_p$. Whenever we use indices, they should be interpreted in a cyclic fashion (i.e., such that $v_i = v_{i+p}$ for all $i \in \mathbb{Z}$).

\section{Graph generation algorithm}\label{sec:algo}
The results in Section~\ref{sec:inducedC7} were partially obtained with the help of a computer. This is a popular trend in the study of $k$-vertex-critical $\mathcal{F}$-free graphs, see for example~\cite{J25}. More precisely, the key algorithm that we use is based on the algorithm that was initially proposed in~\cite{HKLSS10} and later improved in a series of papers by introducing additional pruning rules, expansion rules, heuristics and generalizing several concepts (see~\cite{GS18,GJORS24,XJGH23,XJGH24}). For completeness, we will explain the main ideas of this algorithm in the current section, but we refer the interested reader to the aforementioned references for additional details (e.g.\ the formal correctness proof and details on the various sanity checks that were used to help verify that the implementation does not contain any bugs). We make the source code of the precise version of the algorithm that was used in the current paper publicly available at~\cite{GitHub}.

The algorithm is a graph generation algorithm that recursively generates all $k$-vertex-critical $\mathcal{F}$-free graphs on at most $n$ vertices that contain the graph $H$ as an induced subgraph (where $k$, $n$, $\mathcal{F}$ and $H$ are parameters of the algorithm). The pseudocode of the algorithm is given in Algorithm~\ref{algo:genGraphs}. In each recursion step, the algorithm will select an expansion rule and expand the current graph $G$ by adding a single vertex $u$ and adding edges between $u$ and vertices in $V(G)$ in all possible ways that are permitted by the expansion rule. For example, if $G$ is an induced subgraph of a $k$-vertex-critical graph $G'$ and there is a vertex $v \in V(G)$ with strictly fewer than $k-1$ neighbors in $V(G)$, then we know that $v$ must have at least one more neighbor in $V(G') \setminus V(G)$, since every vertex in a $k$-vertex-critical graph has degree at least $k-1$. Therefore, one possible expansion rule could be to choose a vertex $v \in V(G)$ with strictly fewer than $k-1$ neighbors in $V(G)$ and add a vertex $u$ and edges between $u$ and vertices in $V(G)$ in all possible ways such that $v$ and $u$ are adjacent. Apart from expansion rules, the algorithm also uses pruning rules that allow the algorithm to backtrack without having to expand the current graph. For example, if the current graph is not $\mathcal{F}$-free, then it can clearly not occur as an induced subgraph of an $\mathcal{F}$-free $k$-vertex-critical graph and therefore the algorithm does not have to expand the current graph. It is important to note that for certain parameters $k$, $n$, $\mathcal{F}$ and $H$, the algorithm will be able to prune all graphs with strictly fewer than $n$ vertices from the search space. In such cases, one can conclude that there are only finitely many $\mathcal{F}$-free $k$-vertex-critical graphs that contain $H$ as an induced subgraph and we will repeatedly rely on this observation. In the current paper, we will set $n$ to be infinity and then one can conclude that there are finitely many $\mathcal{F}$-free $k$-vertex-critical graphs that contain $H$ as an induced subgraph if the algorithm terminates when called with the parameters $k$, $n=\infty$, $\mathcal{F}$ and $H$. The algorithm uses several such expansion rules and pruning rules (that are often much more complicated than the simple rules we just discussed). For brevity, we do not explain all of these rules in the current paper, but instead refer the interested reader to the references described in the first paragraph of this section, where these rules were originally introduced and discussed in detail.

\begin{algorithm}[ht!b]
\caption{generateGraphs(Integer $k$, Integer $n$, Set of graphs $\mathcal{F}$, Induced graph $H$)}
\label{algo:genGraphs}
  \begin{algorithmic}[1]
		\STATE // Generate and output all $k$-vertex-critical $\mathcal{F}$-free graphs on at most $n$ vertices that contain $H$ as an induced subgraph.
        \STATE // Return the largest order of a graph that the algorithm was not able to prune. If this is strictly smaller than $n$, this means that there are only finitely many $k$-vertex-critical $\mathcal{F}$-free graphs that contain $H$ as an induced subgraph.
        \IF{$H$ can be pruned by one of the pruning rules}
            \STATE return $|V(H)|-1$
        \ENDIF
        
        \STATE // Since $H$ was not pruned, we may now assume that $H$ is $\mathcal{F}$-free
        \STATE // $H$ might be a proper induced subgraph of a $k$-vertex-critical $\mathcal{F}$-free graph in case $H$ admits a proper $(k-1)$-coloring
        \IF{$H$ admits a proper $(k-1)$-coloring}
            \STATE $\text{largestOrder} \gets |V(H)|$

            \IF{$|V(H)|<n$}
                \STATE $R \gets \text{chooseExpansionRule()}$
                \FOR{every graph $H'$ obtained by adding a vertex $u$ to $H$ and edges between $u$ and vertices in $V(H)$ in all possible ways permitted by the expansion rule $R$}
                    \STATE $\text{largestOrder} \gets \max(\text{largestOrder},\text{generateGraphs}(k,n,H',\mathcal{F}))$
                \ENDFOR
            \ENDIF
            
            \STATE return $\text{largestOrder}$
        \ELSE
            \IF{$H$ is $k$-vertex-critical}
                \STATE Output $H$
            \ENDIF
            \STATE return $|V(H)|$
        \ENDIF

  \end{algorithmic}
\end{algorithm}

\section{4-vertex-critical $\{P_7,C_3\}$-free graphs with an induced $C_7$ and 3-colorability of $\{P_5+P_1,C_3\}$-free graphs}\label{sec:inducedC7}

The main theorem of this section characterizes all 4-vertex-critical $\{P_7,C_3\}$-free graphs with an induced $C_7$.
\begin{theorem}\label{thm:contain C_7}
    There are exactly three 4-vertex-critical $\{P_7,C_3\}$-free graphs containing an induced $C_7$, namely graphs (d), (e) and (f) from Fig.~\ref{fig:7Graphs}.
\end{theorem}
\begin{proof}
    Suppose for the sake of obtaining a contradiction that $G$ is a 4-vertex-critical $\{P_7,C_3\}$-free graph different from the graphs (d), (e) and (f) from Fig.~\ref{fig:7Graphs} and $G$ contains an induced $C_7$. We will obtain a contradiction by showing that $G$ is 3-colorable. 
    Let $C:=v_1-v_2-v_3-v_4-v_5-v_6-v_7-v_1$ be an induced $C_7$ in $G$. We define the following sets
    \begin{itemize}
        \item $S_0$ is the set of vertices that have no neighbors in $C$; 
        \item For each $i \in [7]$, $S_1(i)$ is the set of vertices whose neighbors in $C$ are exactly $v_{i-1},v_{i+1}$; 
        \item For each $i \in [7]$, $S_2(i)$ is the set of vertices whose neighbors in $C$ are exactly $v_{i-2},v_{i+2}$; 
        \item For each $i \in [7]$, $S_3(i)$ is the set of vertices whose neighbors in $C$ are exactly $v_{i-2},v_i,v_{i+2}$.
    \end{itemize}
    Set $S_j=\bigcup_{i=1}^{7} S_j(i)$ for each $j\in \{1,2,3\}$. 

    \begin{claim}\label{claim:neighbors of C}
        $N(C)=S_1\cup S_2\cup S_3$.
    \end{claim}
    \begin{proof}
        Let $v$ be a vertex in $N(C)$.
        Since $G$ is triangle-free, $|N(v)\cap V(C)|\leq 3$. 
        If $|N(v)\cap V(C)|=1$, say $vv_1\in E(G)$, then $v-v_1-v_2-v_3-v_4-v_5-v_6$ is an induced $P_7$. 
        If $|N(v)\cap V(C)|=2$, then again since $G$ is triangle-free, there is an index $i\in [7]$ such that either $N(v)\cap V(C)=\{v_{i-1},v_{i+1}\}$ or $N(v)\cap V(C)=\{v_{i-2},v_{i+2}\}$, and hence $v\in S_1\cup S_2$. 
        Again using the fact that $G$ is triangle-free, if $|N(v)\cap V(C)|=3$, it follows that there is an index $i\in [7]$ such that $N(v)\cap V(C)=\{v_{i-2},v_i,v_{i+2}\}$ and hence $v\in S_3$. 
        This completes the proof of Claim \ref{claim:neighbors of C}.
    \end{proof}

    \begin{figure}[h!]
\begin{center}
\subcaptionbox{$L_1$}{\begin{tikzpicture}[scale=0.7]
  \def\sides{7}
  \def\radius{3}

  \fill (0,0) circle (2.61pt);  
  \foreach \i in {1,...,\sides} {
    \fill ({360/\sides * \i-12.86}:\radius) circle (2.61pt);

    \draw ({360/\sides * (\i + 1)-12.86}:\radius) -- ({360/\sides * \i-12.86}:\radius);
  }
  \draw (0,0) -- ({360/\sides * 4-12.86}:\radius);
  \draw (0,0) -- ({360/\sides * 0-12.86}:\radius);
  
\end{tikzpicture}} \quad
\subcaptionbox{$L_2$}{\begin{tikzpicture}[scale=0.7]
  \def\sides{7}
  \def\radius{3}

  \fill (0,0) circle (2.61pt);
  \fill (0,-1.2) circle (2.61pt);
  \foreach \i in {1,...,\sides} {
    \fill ({360/\sides * \i-12.86}:\radius) circle (2.61pt);

    \draw ({360/\sides * (\i + 1)-12.86}:\radius) -- ({360/\sides * \i-12.86}:\radius);
  }
  \draw (0,0) -- ({360/\sides * 4-12.86}:\radius);
  \draw (0,0) -- ({360/\sides * 2-12.86}:\radius);
  \draw (0,0) -- ({360/\sides * 0-12.86}:\radius);
  \draw (0,0) -- (0,-1.2);
\end{tikzpicture}}
\end{center}
\caption{The graphs $L_1$ and $L_2$.}\label{fig:2Graphs} 
\end{figure}
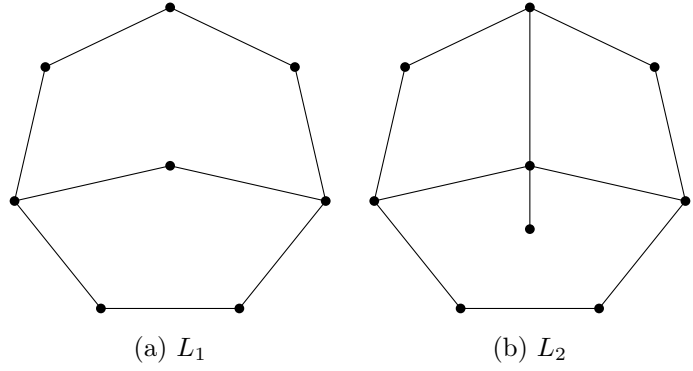

    For each $j\in \{0,1,2,3\}$ and a vertex $v$, if $v\in S_j$, then we call $v$ an $S_j$-vertex. Let $L_1$ be the graph shown in Fig.~\ref{fig:2Graphs} (a) (i.e., the graph consisting of an induced $C_7$ and an $S_2$-vertex). By running Algorithm~\ref{algo:genGraphs} using the parameters $k=4$, $n=\infty$, $\mathcal{F}=\{P_7,C_3\}$ and $H=L_1$, one obtains that there are exactly three 4-vertex-critical $\{P_7,C_3\}$-free graphs containing an induced $L_1$ (namely, graphs (d), (e) and (f) from Fig.~\ref{fig:7Graphs}). So we may assume that $S_2=\emptyset$. 
    Thus, $N(C)=S_1\cup S_3$.

    Let $L_2$ be the graph shown in Fig.~\ref{fig:2Graphs} (b) (i.e., the graph consisting of an induced $C_7$ and an $S_3$-vertex adjacent to an $S_0$-vertex). By running Algorithm~\ref{algo:genGraphs} using the parameters $k=4$, $n=\infty$, $\mathcal{F}=\{P_7,C_3\}$ and $H=L_2$, one obtains that there is exactly one 4-vertex-critical $\{P_7,C_3\}$-free graph containing an induced $L_2$ (namely, graph (e) from Fig.~\ref{fig:7Graphs}). So we may assume that for every $v\in S_0$, $N(v)\cap N(C)\subseteq S_1$.  
    However, if an $S_0$-vertex $v$ has a neighbor $u\in S_1(i)$, then $v-u-v_{i+1}-v_{i+2}-v_{i+3}-v_{i-3}-v_{i-2}$ is an induced $P_7$. 
    Thus, $N(S_0)\cap N(C)=\emptyset$. 
    It follows that $S_0=\emptyset$ (since $G$ is connected) and 
    \begin{equation}\label{eqn-1} 
        V(G)=V(C)\cup S_1\cup S_3.
    \end{equation}

    We call a graph $H$ a {\em heptagram} if $H$ can be partitioned into seven stable sets, say $\{H_1,H_2,\ldots,H_7\}$, such that for each $v\in H_i$, where $i\in [7]$, $v$ has a neighbor in $H_{i-1}$ and a neighbor in $H_{i+1}$ but $v$ has no neighbors in $H_{i\pm 2}\cup H_{i\pm 3}$. Let $H:=\{H_1,H_2,\ldots,H_7\}$ be an inclusionwise maximal induced heptagram in $G$. Since $G$ contains an induced $C_7$, such an $H$ exists. 

    \begin{claim}\label{claim:C_7 property in H}
        Every vertex in $H$ is contained in an induced $C_7$ in $H$.
    \end{claim}
    \begin{proof}
        Let $v$ be a vertex in $H_i$ for some $i\in [7]$. 
        By the definition of $H$, $v$ has a neighbor $v_{i+1}\in H_{i+1}$. 
        And similarly, $v_{i+1}$ has a neighbor $v_{i+2}\in H_{i+2}$, $v_{i+2}$ has a neighbor $v_{i+3}\in H_{i+3}$, $v_{i+3}$ has a neighbor $v_{i-3}\in H_{i-3}$, $v_{i-3}$ has a neighbor $v_{i-2}\in H_{i-2}$ and $v_{i-2}$ has a neighbor $v_{i-1}\in H_{i-1}$. 
        If $v$ is not adjacent to $v_{i-1}$, then $v-v_{i+1}-v_{i+2}-v_{i+3}-v_{i-3}-v_{i-2}-v_{i-1}$ is an induced $P_7$, a contradiction. 
        So we may assume that $vv_{i-1}\in E(G)$. 
        It follows that $v-v_{i+1}-v_{i+2}-v_{i+3}-v_{i-3}-v_{i-2}-v_{i-1}-v$ is an induced $C_7$ in $H$. 
        This completes the proof of Claim \ref{claim:C_7 property in H}.
    \end{proof}
    
    \begin{claim}\label{claim:complete property in H}
        For every $i\in [7]$, $H_i$ is complete to $H_{i-1}\cup H_{i+1}$.
    \end{claim}
    \begin{proof}
        Let $v_i$ be a vertex in $H_i$. 
        By symmetry, we only need to prove that $v_i$ is complete to $H_{i+1}$. 
        Suppose to the contrary that there is a vertex $u_{i+1}\in H_{i+1}$ with $v_iu_{i+1}\notin E(G)$. 
        By Claim \ref{claim:C_7 property in H}, let $v_i-v_{i+1}-v_{i+2}-v_{i+3}-v_{i-3}-v_{i-2}-v_{i-1}-v_i$ be an induced $C_7$ in $H$ containing $v_i$. 
        If $u_{i+1}v_{i+2}\in E(G)$, then $u_{i+1}-v_{i+2}-v_{i+3}-v_{i-3}-v_{i-2}-v_{i-1}-v_i$ is an induced $P_7$, a contradiction. 
        So we may assume that $u_{i+1}v_{i+2}\notin E(G)$. 
        By definition of $H$, let $u_{i+2}$ be a neighbor of $u_{i+1}$ in $H_{i+2}$. 
        If $v_{i+1}u_{i+2}\in E(G)$, then $v_{i-3}-v_{i-2}-v_{i-1}-v_i-v_{i+1}-u_{i+2}-u_{i+1}$ is an induced $P_7$, a contradiction. 
        So $v_{i+1}u_{i+2}\notin E(G)$. 
        By replacing $v_i, u_{i+1}$ with $v_{i+1}, u_{i+2}$, we obtain that $u_{i+2}$ is not adjacent to $v_{i+3}$ and $u_{i+2}$ has a neighbor $u_{i+3}\in H_{i+3}$ that is not adjacent to $v_{i+2}$. 
        By repeating this step for $i-3,i-2$ and $i-1$, we obtain $u_{i-3}\in H_{i-3}$, $u_{i-2}\in H_{i-2}$, $u_{i-1}\in H_{i-1}$ and $u_i\in H_i$ such that $\{u_{i-3},u_{i-2},u_{i-1},u_i\}$ is anti-complete to $\{v_i,v_{i+1},v_{i+2},v_{i+3},v_{i-3},v_{i-2},v_{i-1}\}$ and $u_{i+3}-u_{i-3}-u_{i-2}-u_{i-1}-u_i$ is an induced $P_5$. 
        Note that if $u_i$ is not adjacent to $u_{i+1}$, then $u_{i+1}-u_{i+2}-u_{i+3}-u_{i-3}-u_{i-2}-u_{i-1}-u_i$ is an induced $P_7$, a contradiction. 
        So we may assume that $u_iu_{i+1}\in E(G)$. 
        It follows that $u_i-u_{i+1}-u_{i+2}-u_{i+3}-u_{i-3}-u_{i-2}-u_{i-1}-u_i$ is an induced $C_7$ anti-complete to $v_i-v_{i+1}-v_{i+2}-v_{i+3}-v_{i-3}-v_{i-2}-v_{i-1}-v_i$. 
        Since $G$ is connected, there is an induced path $P=x-w_1-w_2-\ldots
        -w_\ell-y$ (note that $x=y$ is allowed) such that only $x$ has neighbors in $v_i-v_{i+1}-v_{i+2}-v_{i+3}-v_{i-3}-v_{i-2}-v_{i-1}-v_i$ and only $y$ has neighbors in $u_i-u_{i+1}-u_{i+2}-u_{i+3}-u_{i-3}-u_{i-2}-u_{i-1}-u_i$ (by considering an arbitrary shortest path between these two cycles). 
        By symmetry, we may assume that $x$ is adjacent to $v_j$ but not to $v_{j-1},v_{j-2}$ and $y$ is adjacent to $u_k$ but not to $u_{k+1}, u_{k+2}$ for some $j,k\in [7]$. 
        Then $v_{j-2}-v_{j-1}-v_j-x-w_1-w_2-\ldots-w_\ell-y-u_k-u_{k+1}-u_{k+2}$ is an induced path with length at least 7, a contradiction. 
        This completes the proof of Claim \ref{claim:complete property in H}. 
    \end{proof}

    Suppose there is a vertex $v\in V(G)\setminus V(H)$. 
    By (1) and Claim \ref{claim:complete property in H}, there is an index $i\in[7]$ such that the neighbors of $v$ in $H$ are either in (and only in) both $H_{i-1}$ and $H_{i+1}$ or in (and only in) both $H_{i-2}$, $H_{i}$ and $H_{i+2}$. 
    In the first case, we could add $v$ into $H_i$ to obtain a larger heptagram, which contradicts the maximality of $H$. 
    So we may assume that for every $v\in V(G)\setminus V(H)$, there is an index $i\in[7]$ such that the neighbors of $v$ in $H$ are in (and only in) both $H_{i-2}$, $H_{i}$ and $H_{i+2}$. 
    
    For each $i\in [7]$, let $T_i$ be the set containing those vertices $v \in V(G)\setminus V(H)$ for which the neighbors of $v$ in $H$ are in (and only in) both $H_{i-2}$, $H_{i}$ and $H_{i+2}$. 
    Thus, $V(G)\setminus V(H)=\bigcup_{i=1}^7 T_i$. 

    \begin{claim}\label{claim:T_i property}
        For every $i\in [7]$ and $v\in T_i$, $v$ is complete to $H_i$ and complete to at least one of $H_{i-2}$ and $H_{i+2}$. 
    \end{claim}
    \begin{proof}
        For every $j\in [7]$, let $v_{j}$ be a vertex in $H_{j}$. 
        Without loss of generality, we may assume that $vv_i,vv_{i-2},vv_{i+2}\in E(G)$. 
        Suppose first that there is a vertex $u_i\in H_i$ which is not adjacent to $v$. 
        By Claim \ref{claim:complete property in H}, $u_i-v_{i-1}-v_{i}-v-v_{i+2}-v_{i+3}-v_{i-3}$ is an induced $P_7$, a contradiction. 
        Then we suppose that there are $u_{i-2}\in H_{i-2}$ and $u_{i+2}\in H_{i+2}$ which are not adjacent to $v$. 
        By Claim \ref{claim:complete property in H} again, $u_{i-2}-v_{i-1}-v_{i-2}-v-v_{i+2}-v_{i+1}-u_{i+2}$ is an induced $P_7$, a contradiction. 
        This completes the proof of Claim \ref{claim:T_i property}.
    \end{proof}

    By Claim \ref{claim:T_i property}, $T_i$ is complete to $H_i$. 
    Since $G$ is triangle-free, $T_i$ is a stable set and is anti-complete to $T_{i\pm 2}$. 
    For a fixed $i\in [7]$, let $T_i^-$ be the set of vertices in $T_i$ that are complete to $H_{i-2}$ and $T_i^+=T_i\setminus T_i^-$. 
    By Claim \ref{claim:T_i property} again, $T_i^+$ is complete to $H_{i+2}$. 
    Since $G$ is triangle-free, $T_i^-$ is anti-complete to $T_{i+3}$ and $T_i^+$ is anti-complete to $T_{i-3}$. 

    Let $f:V(G)\rightarrow \{1,2,3\}$ be a function such that $f(v)=1$ for every $v\in H_{i-2}\cup H_i\cup H_{i+2}\cup T_{i-1}\cup T_{i+1}$, $f(v)=2$ for every $v\in H_{i-3}\cup H_{i-1}\cup T_{i-2}\cup T_{i+3}\cup T_{i}^-$ and $f(v)=3$ for every $v\in H_{i+1}\cup H_{i+3}\cup T_{i-3}\cup T_{i+2}\cup T_{i}^+$. 
    It follows that $f$ is a proper 3-coloring of $G$, see Fig.~\ref{fig:placeholder} as an illustration when $i=1$. 
    This completes the proof of Theorem \ref{thm:contain C_7}.

    \begin{figure}[h!]
    \centering

\scalebox{0.8}{
\begin{tikzpicture}[
    node distance=2.5cm,
    every node/.style={
        circle, 
        draw=black, 
        thick, 
        minimum size=12mm,
        font=\normalsize\sffamily,
        inner sep=1pt 
    },
    every edge/.style={
        draw, 
        very thick,  
        shorten >=1pt,  
        shorten <=1pt  
    },
    scale=1
]

\foreach \i in {0,1,...,6} {
    \pgfmathsetmacro{\angle}{90 - \i*360/7} 
    \pgfmathsetmacro{\x}{5*cos(\angle)}
    \pgfmathsetmacro{\y}{5*sin(\angle)}
    \node (v\i) at (\x, \y) {};
}

\foreach \i in {7,8,9,10,11,12,13} {
    \pgfmathsetmacro{\angle}{90 - (\i-7)*360/7} 
    \pgfmathsetmacro{\x}{2.5*cos(\angle)}
    \pgfmathsetmacro{\y}{2.5*sin(\angle)}
    \node (v\i) at (\x, \y) {};
}

\node[fill=red!70] at (v0) {$H_1$};
\node[fill=yellow!70] at (v1) {$H_2$};
\node[fill=red!70] at (v2) {$H_3$};
\node[fill=yellow!70] at (v3) {$H_4$};
\node[fill=blue!50] at (v4) {$H_5$};
\node[fill=red!70] at (v5) {$H_6$};
\node[fill=blue!50] at (v6) {$H_7$};

\fill[yellow!70] (v7) -- ++(0,0.6) arc (90:270:0.6) -- cycle;
\fill[blue!50] (v7) -- ++(0,0.6) arc (90:-90:0.6) -- cycle;

\draw [thick](v7) circle (0.6cm);

\node at (v7) [scale=0.9,draw=none] {$\{6,1,3\}$};

\node [fill=red!70] at (v8) [scale=0.9]{$\{7,2,4\}$};
\node [fill=yellow!70] at (v9) [scale=0.9]{$\{1,3,5\}$};
\node [fill=blue!50] at (v10) [scale=0.9]{$\{2,4,6\}$};
\node [fill=yellow!70] at (v11) [scale=0.9]{$\{3,5,7\}$};
\node [fill=blue!50] at (v12) [scale=0.9]{$\{4,6,1\}$};
\node [fill=red!70] at (v13) [scale=0.9]{$\{5,7,2\}$};

\draw [ultra thick](v1) -- (v2);
\draw [ultra thick](v3) -- (v2);
\draw [ultra thick](v3) -- (v4);
\draw [ultra thick](v4) -- (v5);
\draw [ultra thick](v5) -- (v6);
\draw [ultra thick](v6) -- (v0);
\draw [ultra thick](v1) -- (v0);


\draw[dashed] (v7) -- (v10);
\draw[dashed] (v7) -- (v11);

\draw[dashed] (v7) -- (v9);
\draw[dashed] (v8) -- (v10);
\draw[dashed] (v9) -- (v11);
\draw[dashed] (v10) -- (v12);
\draw[dashed] (v11) -- (v13);
\draw[dashed] (v12) -- (v7);
\draw[dashed] (v13) -- (v8);

\end{tikzpicture}
}
    \caption{Illustration of a proper 3-coloring of $G$. The dashed lines indicate sets that are anti-complete to each other. Each set in the middle with label $\{a,b,c\}$ indicates $T_b$. The yellow part of the set with label $\{6,1,3\}$ indicates that $T_1^+$ should be colored yellow and the blue part indicates that $T_1^-$ should be colored blue.}
    \label{fig:placeholder}
\end{figure}
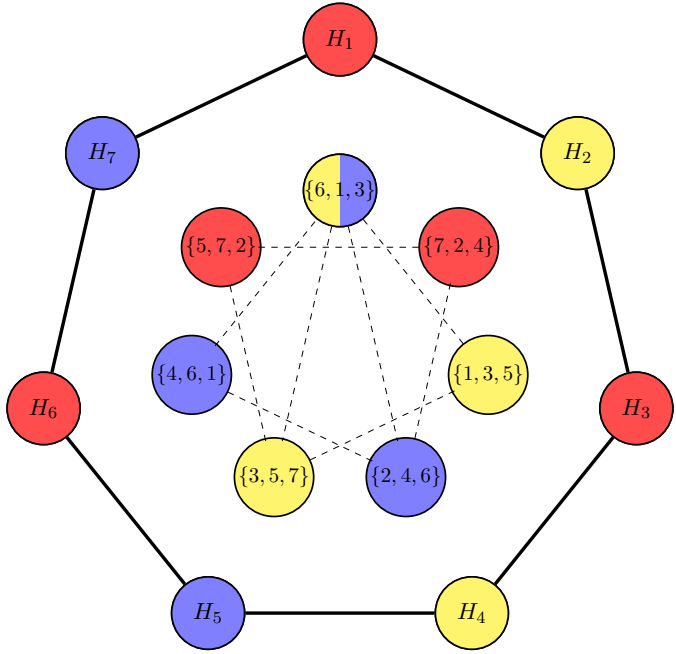
\end{proof}

Every graph that is not $3$-colorable must contain an odd cycle. When running Algorithm~\ref{algo:genGraphs} using the parameters $k=4$, $n=\infty$, $\mathcal{F}=\{P_5+P_1,C_3,C_7\}$ and $H=C_5$, the algorithm terminates without finding any $4$-vertex-critical graphs. We therefore obtain the following proposition.
\begin{proposition}
\label{prop:P5PlusP1C3C7Free}
All $\{P_5+P_1,C_3,C_7\}$-free graphs are $3$-colorable.
\end{proposition}

The three graphs from Theorem~\ref{thm:contain C_7} all contain an induced $P_5+P_1$. Hence, by combining Theorem~\ref{thm:contain C_7} and Proposition~\ref{prop:P5PlusP1C3C7Free} we obtain the following corollary.

\begin{corollary}
All $\{P_5+P_1,C_3\}$-free graphs are $3$-colorable.
\end{corollary}

\section{4-vertex-critical $\{4K_{2},C_3\}$-free graphs}\label{sec:infFam}
For any integer $k\geq 1$, let $G_k$ be the circulant graph with vertex set $\{v_1,v_2,\ldots,v_{3k+10}\}$ such that vertex $v_i$ has the following neighbors: $\{v_{i-1},v_{i+1}\}\cup \{v_{i+5},v_{i+8},\ldots, v_{i+5+3k}\}$. For example, Fig.~\ref{fig:G1} visualizes the graph $G_1$. For any two vertices $u,v\in V(G_k)$, let $M(u,v)$ be the set of common non-neighbors of $u$ and $v$.  
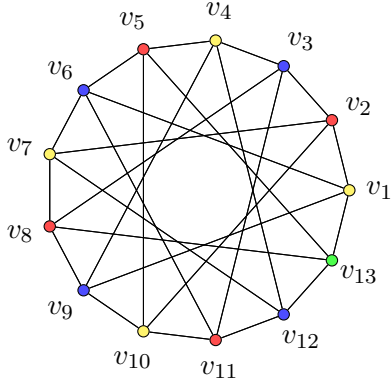
\begin{figure}[h!]
    \centering
    \begin{tikzpicture}[scale=1]
    \node[circle, draw, fill=yellow!70, inner sep=1.5pt] (v0) at (0:2) {};
    \node at (0:2.4) {$v_1$};
    
    \node[circle, draw, fill=red!70, inner sep=1.5pt] (v1) at (27.69:2) {};
    \node at (27.69:2.4) {$v_2$};
    
    \node[circle, draw, fill=blue!70, inner sep=1.5pt] (v2) at (55.38:2) {};
    \node at (55.38:2.4) {$v_3$};
    
    \node[circle, draw, fill=yellow!70, inner sep=1.5pt] (v3) at (83.08:2) {};
    \node at (83.08:2.4) {$v_4$};
    
    \node[circle, draw, fill=red!70, inner sep=1.5pt] (v4) at (110.77:2) {};
    \node at (110.77:2.4) {$v_5$};
    
    \node[circle, draw, fill=blue!70, inner sep=1.5pt] (v5) at (138.46:2) {};
    \node at (138.46:2.4) {$v_6$};
    
    \node[circle, draw, fill=yellow!70, inner sep=1.5pt] (v6) at (166.15:2) {};
    \node at (166.15:2.4) {$v_7$};
    
    \node[circle, draw, fill=red!70, inner sep=1.5pt] (v7) at (193.85:2) {};
    \node at (193.85:2.4) {$v_8$};
    
    \node[circle, draw, fill=blue!70, inner sep=1.5pt] (v8) at (221.54:2) {};
    \node at (221.54:2.4) {$v_9$};
    
    \node[circle, draw, fill=yellow!70, inner sep=1.5pt] (v9) at (249.23:2) {};
    \node at (249.23:2.4) {$v_{10}$};
    
    \node[circle, draw, fill=red!70, inner sep=1.5pt] (v10) at (276.92:2) {};
    \node at (276.92:2.4) {$v_{11}$};
    
    \node[circle, draw, fill=blue!70, inner sep=1.5pt] (v11) at (304.62:2) {};
    \node at (304.62:2.4) {$v_{12}$};
    
    \node[circle, draw, fill=green!70, inner sep=1.5pt] (v12) at (332.31:2) {};
    \node at (332.31:2.4) {$v_{13}$};

    \draw (v0) to (v1);
    \draw (v0) to (v12);
    \draw (v0) to (v5);
    \draw (v0) to (v8);
    
    \draw (v1) to (v2);
    \draw (v1) to (v0);
    \draw (v1) to (v6);
    \draw (v1) to (v9);
    
    \draw (v2) to (v3);
    \draw (v2) to (v1);
    \draw (v2) to (v7);
    \draw (v2) to (v10);
    
    \draw (v3) to (v4);
    \draw (v3) to (v2);
    \draw (v3) to (v8);
    \draw (v3) to (v11);
    
    \draw (v4) to (v5);
    \draw (v4) to (v3);
    \draw (v4) to (v9);
    \draw (v4) to (v12);
    
    \draw (v5) to (v6);
    \draw (v5) to (v4);
    \draw (v5) to (v10);
    \draw (v5) to (v0);
    
    \draw (v6) to (v7);
    \draw (v6) to (v5);
    \draw (v6) to (v11);
    \draw (v6) to (v1);
    
    \draw (v7) to (v8);
    \draw (v7) to (v6);
    \draw (v7) to (v12);
    \draw (v7) to (v2);
    
    \draw (v8) to (v9);
    \draw (v8) to (v7);
    \draw (v8) to (v0);
    \draw (v8) to (v3);
    
    \draw (v9) to (v10);
    \draw (v9) to (v8);
    \draw (v9) to (v1);
    \draw (v9) to (v4);
    
    \draw (v10) to (v11);
    \draw (v10) to (v9);
    \draw (v10) to (v2);
    \draw (v10) to (v5);
    
    \draw (v11) to (v12);
    \draw (v11) to (v10);
    \draw (v11) to (v3);
    \draw (v11) to (v6);
    
    \draw (v12) to (v0);
    \draw (v12) to (v11);
    \draw (v12) to (v4);
    \draw (v12) to (v7);
\end{tikzpicture}
    \caption{$G_1$ and a proper $4$-coloring of $G_1$.}
    \label{fig:G1}
\end{figure}

We show that there are infinitely many 4-vertex-critical $\{4K_2,C_3\}$-free graphs.
\begin{theorem}\label{P11}
    For each $k\geq 3$, $G_k$ is a 4-vertex-critical $\{4K_2,C_3\}$-free graph. 
\end{theorem}
\begin{proof}
    We start our proof by showing that $G_k$ is $\{4K_2,C_3\}$-free. 
    \begin{claim}\label{claim:P_11 C_3-free}
        For every $k\geq 3$, $G_k$ is $\{4K_2,C_3\}$-free. 
    \end{claim}
    \begin{proof}
        $G_k$ is clearly $C_3$-free because of the definition of $G_k$. Therefore, we only need to prove that $G_k$ is $4K_2$-free. 
        Suppose for the sake of obtaining a contradiction that $G_k$ contains an induced $4K_2$, say $H$.  
        We suppose first that $v_iv_{i+1}$ is an edge of $H$ for some $i$.  
        Note that $M(v_i,v_{i+1})=\{v_{i+3}\}\cup \{v_{i+4},v_{i+7},v_{i+10},\ldots, v_{i+3k+4}, v_{i+3k+7}\}\cup \{v_{i-2}\}$. 
        Since $H$ is a $4K_2$, $G[M(v_i,v_{i+1})]$ contains a $3K_2$. 
        However, note that $M(v_i,v_{i+1}) \setminus \{v_{i+3},v_{i-2}\} =\{v_{i+4},v_{i+7},v_{i+10},\ldots, v_{i+3k+4}, v_{i+3k+7}\}$ is a stable set and therefore $G[M(v_i,v_{i+1})]$ is $3K_2$-free, a contradiction. So we may assume that for any $i\in [3k+10]$, $v_iv_{i+1}$ is not an edge of $H$. 
        
        Suppose that $v_iv_{i+5+3j}$ is an edge of $H$, where $0\leq j\leq k$. 
        Note that  
	\begin{equation}
		\begin{aligned}
			M(v_i,v_{i+5+3j})=\{v_{i+2}\}\cup A\cup \{v_{i+3+3j}\} \cup \{v_{i+7+3j}\}\cup B\cup \{v_{i-2}\},
		\end{aligned}
	\end{equation}
	where $$A=\{v_{i+4},v_{i+7},v_{i+10},\ldots, v_{i+1+3j}\}, B=\{v_{i+9+3j},v_{i+12+3j},\ldots, v_{i+9+3(k-1)}\}.$$ 
	Note that $A=\emptyset$ if $j=0$ and $B=\emptyset$ if $j=k$. 
    Since $A$ and $B$ are both stable sets and $A$ is complete to $B$, $|V(H)\cap (A\cup B)|\leq 3$. 
	Thus, $|V(H)\cap \{v_{i+2},v_{i+3+3j},v_{i+7+3j},v_{i-2}\}|\geq 3$. 
	So by symmetry, we may assume that $v_{i+2},v_{i+7+3j}\in V(H)$. Note that $v_{i+2}$ is complete to $A\setminus \{v_{i+4}\}$ and $v_{i+7+3j}$ is complete to $B\setminus \{v_{i+9+3j}\}$. 
    It follows that $V(H)\cap (A\cup B)\subseteq \{v_{i+4},v_{i+9+3j}\}$. 
    On the other hand, since $|V(H)\cap M(v_i,v_{i+5+3j})|=6$, $|V(H)\cap (A\cup B)|\geq 2$. 
    Thus, $V(H)\cap (A\cup B)=\{v_{i+4},v_{i+9+3j}\}$. 
    So $V(H)\cap M(v_i,v_{i+5+3j})=\{v_{i+2}, v_{i+4}, v_{i+3+3j}, v_{i+7+3j}, v_{i+9+3j}, v_{i-2}\}$. 
    Note that $v_{i-2}v_{i+3+3j},v_{i+4}v_{i+9+3j}\in E(G_k)$. 
    Since $k\geq 3$, either $j\neq 1$ or $j\neq k-1$. 
    If $j\neq 1$, then $v_{i+3+3j}v_{i+4}\in E(G_k)$. 
    If $j\neq k-1$, then $v_{i-2}v_{i+9+3j}\in E(G_k)$. 
    Each case contradicts that $H$ is an induced $4K_2$. 
    This completes the proof of Claim \ref{claim:P_11 C_3-free}.
    \end{proof}

    Next, we show that every graph obtained by deleting one vertex from $G_k$ admits a proper 3-coloring. 
    \begin{claim}\label{claim:critical of G_k}
        For every $v\in V(G_k)$, $\chi(G_k\setminus \{v\})\leq 3$. 
    \end{claim}
    \begin{proof}
        Let $c: V(G_k\setminus \{v_1\}) \to \{0,1,2\}$ be a function such that $c(v_i)\equiv i\pmod{3}$. 
        We then prove that $c$ is a proper 3-coloring of $G_k\setminus \{v_1\}$. 
        Suppose to the contrary that there are two adjacent vertices $v_i,v_j\in V(G_k\setminus \{v_1\})$ with $i<j$ such that $c(v_i)=c(v_j)$.
        So $j-i \equiv 0\pmod{3}$. 
        By the definition of $G_k$, since $v_i$ and $v_j$ are adjacent, $j=i+1$ or $j=i+5+3\ell$ for some integer $\ell \geq 0$. 
        Hence, $j-i \equiv 1\pmod{3}$ or $j-i \equiv 2\pmod{3}$, which is a contradiction. 
        By the symmetry of $G_k$, for every $v\in V(G)$, $\chi(G\setminus \{v\})\leq 3$. 
        This completes the proof of Claim \ref{claim:critical of G_k}.
    \end{proof}
    By Claim \ref{claim:critical of G_k}, $\chi(G_k)\leq 4$. 
    Suppose now that there is a proper 3-coloring $c$ of $G_k$. 
	Since $\{v_1,v_2,v_3,v_4,v_5\}$ induces a $C_5$, $\chi(G_k)\geq 3$. 
	So there exists an index $i$ such that the colors of $v_i,v_{i+1},v_{i+2}$ are pairwise different. 
	By symmetry, we may assume that $c(v_i)=1,c(v_{i+1})=2$ and $c(v_{i+2})=3$. 
    \begin{claim}\label{claim:3-coloring contradict}
        The colors of $v_{i+1},v_{i+2},v_{i+3}$ are also pairwise different. 
        In particular, $c(v_{i+3})=c(v_i)$.
    \end{claim}
   \begin{proof}
       Suppose to the contrary that $c(v_{i+3})=2$. 
       So $c(v_{i+8})=3$. 
       If $c(v_{i+4})=1$, then $c(v_{i+9})=3$, which is a contradiction. 
       So we may assume that $c(v_{i+4})=3$. 
       Hence, $c(v_{i+9})=1$.  
       If $c(v_{i-1})=3$, then since $c(v_{i+3})=2$, we have $c(v_{i-2})=1$. 
       Then since $v_{i-2}v_{i+9}\in E(G)$ and both of these two vertices are colored 1, we obtain a contradiction.  
       So we may assume that $c(v_{i-1})=2$. But then $c(v_{i+7})=1$, $c(v_{i+6})=3$ and $c(v_{i-2})=1$, which leads again to a contradiction with $c(v_{i+9})=1$. 
       This completes the proof of Claim \ref{claim:3-coloring contradict}.
   \end{proof}
		
	By Claim \ref{claim:3-coloring contradict}, if $j\equiv i\pmod{3}$, then $c(v_j)=c(v_i)$. 
    Thus, $c(v_{i+3k+9})=c(v_{i})=1$, which is a contradiction. 
    It follows that for every $k\geq 3$, $G_k$ is 4-vertex-critical. 
    This completes the proof of Theorem \ref{P11}.
\end{proof}

Since $4K_2$ is an induced subgraph of $P_{11}$, we obtain the following corollary.

\begin{corollary}\label{cor:P11}
    For each $k\geq 3$, $G_k$ is a 4-vertex-critical $\{P_{11},C_3\}$-free graph. 
\end{corollary}

This settles the claim from the introduction that $g(t)=5$ for all $t \geq 11$. We remark that it is not possible to replace $P_{11}$ by $P_{10}$ in Corollary~\ref{cor:P11}, since the set $\{v_{1}, v_{9}, v_{3k+5}, v_3, v_{11}, v_{3k+7}, v_5, v_{13}, v_{3k+9}, v_7\}$ induces a $P_{10}$ in $G_k$ for each $k \geq 4$.



    \section{Conclusion}
In this paper, we presented several contributions for the $3$-coloring problem on $\{F,C_3\}$-free graphs, where $F$ is an induced subgraph of a path. More precisely, we showed that there are exactly three 4-vertex-critical $\{P_7,C_3\}$-free graphs containing an induced $C_7$, that every $\{P_5+P_1,C_3\}$-free graph is $3$-colorable and that there is an infinite family of 4-vertex-critical $\{4K_2,C_3\}$-free (and hence also $P_{11}$-free) graphs. The following natural question sparks our interest:
\begin{question}
\label{ques}
What is the largest integer $t$ for which there are only finitely many 4-vertex-critical $\{P_t,C_3\}$-free graphs?
\end{question}

\subsection*{Acknowledgements}

\noindent  
We acknowledge the support the joint FWO-NSFC
scientific mobility project with grant number VS01224N. The research of Jan Goedgebeur was supported by Internal Funds of KU Leuven and an FWO grant with grant number G0AGX24N. Jorik Jooken is supported by a Postdoctoral Fellowship of the Research Foundation Flanders (FWO) with grant number 1222524N. We are grateful to Ben Cameron for making us aware of the result by Kami\'{n}ski and Pstrucha~\cite{KP19} and for suggesting Question~\ref{ques} in the last section.


\end{document}